\documentclass[11pt]{article}

\usepackage{amsmath,amssymb,amsthm,mathtools}

\numberwithin{equation}{section}

\title{Formation of singularities for a family of one-dimensional quasilinear wave equations beyond the variational case}
\author{
Yuusuke Sugiyama\\
Department of Mathematics, Tokyo University of Science\\
Kagurazaka 1-3, Shinjuku-ku, Tokyo 162-8601, Japan\\
\texttt{sugiyama.y@rs.tus.ac.jp}
}
\date{}

\usepackage{amssymb,amsmath,amsthm,mathrsfs,delarray,enumerate}
\usepackage{graphics}

\theoremstyle{definition} 
\newtheorem{Def}{Deffinition}[section]
\newtheorem{proposition}[Def]{Proposition}
\newtheorem{lemma}[Def]{Lemma}
\newtheorem{theorem}[Def]{Theorem}
\newtheorem{remark}[Def]{Remark}

\newcommand{\N}{\mathbb{N}} 
\newcommand{\R}{\mathbb{R}} 

\makeatletter

\@addtoreset{equation}{section}

\begin{document}
\maketitle %
\begin{abstract}
We consider finite-time singularity formation for classical solutions to the
following parameterized nonlinear wave equation:
\[
  u_{tt}=c(u)^2u_{xx}+\lambda c(u)c'(u)(u_x)^2,
\]
where \(\lambda\in[0,2]\) is a parameter. In previous works, it is known that
finite-time blow-up solutions exist for \(0<\lambda\le1\) and for
\(\lambda=2\).

The cases \(\lambda=1\) and \(\lambda=2\) are known as the variational wave
equation and the \(p\)-system, respectively, and they possess symmetric
structures or conservation laws. For \(0<\lambda<1\), the blow-up construction
is based on the fact that, after decomposing the wave into two Riemann
variables, one component can be kept sufficiently small and hence does not
prevent the other component from blowing up. In the present paper, we treat the
remaining intermediate case \(1<\lambda<2\). This case is essentially different
from the previous one, since both Riemann variables may grow.

The key idea is to introduce suitable supremum functions for the Riemann
variables and to derive a comparison principle through right Dini derivatives.
This allows us to control the relative size of the two components and to obtain
a Riccati-type differential inequality for the dominant supremum function. 
\end{abstract}
%

\section{Introduction}

\subsection{Background and known results}
In the present paper, we consider the Cauchy problem for the following family
of one-dimensional quasilinear wave equations parameterized by \(\lambda\):
\begin{eqnarray} 
\left\{  \begin{array}{ll} \label{req}
   u_{tt} =  c(u)^{2}u_{xx} + \lambda c(u)c'(u)( u_x)^2,  \ \ (t,x) \in (0,T] \times \R, \\   
  u(0,x) = u_0 (x),\ \ x \in \R,                \\            
    \partial_t u(0,x) = u_1 (x), \ \ x \in \R.  
\end{array} \right.  
\end{eqnarray} 
Here \(u(t,x)\) is an unknown real-valued function and
\(c'(\theta)=dc(\theta)/d\theta\). We assume that \(c\) is smooth in a
neighborhood of \(0\) and satisfies
\begin{eqnarray}
c(0) =c_0 >0  \ \mbox{and} \ c'(0)=c_1  >0. \label{non-deg1}
\end{eqnarray}

The parameterized nonlinear wave equation \eqref{req} was introduced by
Glassey, Hunter and Zheng \cite{GHZ2}  (see also Chen and Shen \cite{GCYS}).
Equation \eqref{req} can be rewritten as
\begin{eqnarray*}
u_{tt} - c(u)^{2-\lambda} [ c(u)^{\lambda} u_x ]_x =0,
\end{eqnarray*}
and has different mathematical and physical interpretations depending on
\(\lambda\).

If \(\lambda=2\), then the parameterized nonlinear wave equation
\eqref{req} is formally equivalent to the conservation system
\begin{equation*}
  \partial_t
  \begin{pmatrix}
    U \\ V
  \end{pmatrix}
  -
  \partial_x
  \begin{pmatrix}
    V \\ p(U)
  \end{pmatrix}
  =0,
\end{equation*}
where
\[
  U(t,x)=u(t,x),\qquad
  V(t,x)=\int_{-\infty}^{x} u_t(t,y)\,dy,
  \qquad
  p'(\theta)=c(\theta)^2 .
\]
This conservation system is referred to as a \(p\)-system and describes
several phenomena of wave propagation in nonlinear media, including
electromagnetic waves in transmission lines, shearing motion in
elastic-plastic rods, and one-dimensional gas dynamics (see  Cristescu  \cite{NC}, Zabusky \cite{z}, and Landau and Lifshitz \cite{ll}).
In the study of the \(p\)-system, the following variables, which are called
Riemann variables, play an important role:
\begin{equation} \label{Riev}
  R=u_t+c(u)u_x,\qquad
  S=u_t-c(u)u_x .
\end{equation}
For hyperbolic systems of conservation laws, necessary and sufficient
conditions for the global existence of classical solutions have been
extensively studied. Roughly speaking,  if
\(R\) and \(S\) are non-positive at the initial time, then the corresponding
classical solution exists globally in time. On the other hand, if either
\(R\) or \(S\) takes a positive value at some point initially, then a
singularity, such as a shock, may form in finite time (e.g. Johnson \cite{JJ}, Klainerman and  Majda \cite{km}, Lax \cite{lax} and  Chen, Pan and Zhu \cite{GPZ}).

When \(\lambda=1\), the equation in \eqref{req} is called the
variational wave equation:
\[
  u_{tt}-c(u)\bigl(c(u)u_x\bigr)_x=0.
\]
As its name suggests, this equation  possesses the energy conservation law
\[
  E(t)
  =
  \int_{\R}
  \left\{
    u_t(t,x)^2+c(u(t,x))^2u_x(t,x)^2
  \right\}\,dx
  =
  E(0).
\]
The variational wave equation has physical backgrounds, including nematic
liquid crystals and long waves on a dipole chain in the continuum limit
(see \cite{GHZ2}).
Zhang and Zheng \cite{zz1} showed the existence of global classical solutions
under the assumption that
\[
  R(0,x)\leq0,\qquad S(0,x)\leq0
  \quad (x\in\R).
\]
Glassey, Hunter and Zheng \cite{GHZ} constructed finite-time blow-up
solutions for suitable initial data violating this non-positive condition.

Using the method of Zhang and Zheng \cite{zz1}, the author in \cite{s1}
extended their global existence result to the range \(0\leq\lambda\leq2\)
under the same non-positive condition on the initial Riemann variables.
Under this condition, the non-positivity of \(R\) and \(S\) is preserved.
Moreover, one can obtain time-independent lower bounds for \(R\) and \(S\).
In \cite{s6}, the blow-up result for the variational case \(\lambda=1\) was
extended to the range \(0<\lambda\leq1\).

Recently, the author in \cite{s7} proved the global existence of classical solutions for the case \(\lambda=0\) for arbitrary smooth initial data, without assuming the non-positive condition on both \(R(0,x)\) and \(S(0,x)\).

The purpose of the present paper is to construct a finite-time blow-up solution to \eqref{req} in the remaining case \(1<\lambda<2\), following the ideas of Glassey, Hunter and Zheng \cite{GHZ} and the author's previous work \cite{s6}.

\subsection{Main theorem}

We are now in a position to state the main theorem of the present paper.

\begin{theorem}\label{mainthm}
Let \(1<\lambda<2\). Let \(\phi\in C_0^\infty(\mathbb R)\) satisfy
\[
  \phi'(0)<0
\]
and
\[
  \sup_{x\in\mathbb R}(-\phi'(x))
  >
  \sup_{x\in\mathbb R}\phi'(x).
\]
For \(\delta>0\) and \(\varepsilon>0\), we choose the initial data as
\[
  u_0(x)=\delta\phi(x/\varepsilon),
  \qquad
  u_1(x)=-c(u_0(x))\partial_xu_0(x).
\]
Then there exist constants \(\delta_0,C>0\), depending only on
\(c,\lambda,\phi\), such that, for every \(\varepsilon>0\) and every
\(0<\delta\le \delta_0\), the lifespan \(T^*\) of the corresponding
classical solution is finite and satisfies
\[
  T^* \le C\frac{\varepsilon}{\delta}.
\]

In particular, the gradient blow-up occurs in finite time. Namely,
\[
  \limsup_{t\to T^*}
  \left(
    \|u_t(t)\|_{L^\infty}
    +
    \|u_x(t)\|_{L^\infty}
  \right)
  =
  \infty .
\]
Furthermore, the solution satisfies
\[
  \sup_{[0,T^*)\times\mathbb R}|u(t,x)|
  \le C\delta,
\]
\[
  \int_{\mathbb R}
  \left(
    |\partial_tu(t,x)|^{2/\lambda}
    +
    |\partial_xu(t,x)|^{2/\lambda}
  \right)\,dx
  \le
  C\delta^{2/\lambda}\varepsilon^{1-2/\lambda},
\]
and
\[
  \frac{c_0}{2}
  \le
  c(u(t,x))
  \le
  2c_0
\]
for \((t,x)\in[0,T^*)\times\mathbb R\), where the constant \(C\) is independent of
\(\delta\) and \(\varepsilon\).
\end{theorem}

The following proposition gives a more precise property of the blow-up solution
constructed in Theorem \ref{mainthm}.

\begin{proposition}\label{prop:holder}
The blow-up solution constructed in Theorem \ref{mainthm} satisfies
\[
  |u(t,x)-u(s,y)|
  \le
  C\delta\varepsilon^{\lambda/2-1}
  \left(
    |t-s|^{1-\lambda/2}
    +
    |x-y|^{1-\lambda/2}
  \right)
\]
for all \(x,y\in\mathbb R\) and \(t,s\in[0,T^*)\), where \(C\) is independent
of \(\delta\) and \(\varepsilon\). Furthermore, there exists the limit
\[
  u(T^*,x)=\lim_{t\to T^*}u(t,x)
\]
for all \(x\in\mathbb R\), and the solution satisfies
\[
  |u(T^*,x)-u(T^*,y)|
  \le
  C\delta\varepsilon^{\lambda/2-1}
  |x-y|^{1-\lambda/2}
\]
for all \(x,y\in\mathbb R\).
\end{proposition}

\begin{remark}
It should be noted that the assumptions on the initial data in the present paper
are slightly different from those in Glassey, Hunter and Zheng \cite{GHZ} and
the author's previous work \cite{s6}. In particular, the additional condition
\[
  \sup_{x\in\mathbb R}(-\phi'(x))
  >
  \sup_{x\in\mathbb R}\phi'(x)
\]
was not imposed in those works. Moreover, the previous works used a
one-parameter family of initial data, which corresponds to the special case
\(\delta=\varepsilon\) in the present setting. Introducing two independent
parameters allows us to describe more precisely the dependence of the blow-up
time on the size and the width of the initial data. This is one advantage of the present formulation.
 \end{remark}

\subsection{Difficulty and the idea of proof}
From the equation in \eqref{req}, the variables $R$ and $S$, defined in \eqref{Riev},  satisfy the following:
\begin{eqnarray}\label{fsin}
 \left\{  
\begin{array}{ll}
\partial_t R -c(u)\partial_x R=\dfrac{c'(u)}{4c(u)}( \lambda R^2 + 2(1-\lambda) RS -(2-\lambda )S^2 ), \\
\partial_t S +c(u)\partial_x S =\dfrac{c'(u)}{4c(u)}( \lambda S^2 + 2(1-\lambda) RS -(2-\lambda )R^2 ).
\end{array}
\right.
\end{eqnarray} 
In our proof, the Riemann variable \(S\) is the component which is expected to
blow up in finite time. The initial data are chosen so that
\[
  R(0,x)=0,
\]
and so that the size of \(R\) can be controlled in comparison with that of
\(S\).

There is an essential difference between the cases \(0<\lambda\le1\) and
\(1<\lambda<2\). When \(0<\lambda\le1\), the \(L^{2/\lambda}\)-estimate for
\(R\) and \(S\) is strong enough to control the integral of \(S^2\) along the
characteristic direction of \(R\), since \(2/\lambda\ge2\). In particular, in
Glassey, Hunter and Zheng \cite{GHZ} and the author's previous work
\cite{s6}, this estimate yields a bound for \(\|R(t)\|_{L^\infty}\). As a
result, the negative term
\(
  -(2-\lambda)R^2
\)
in the equation for \(S\), which may prevent the growth of \(S\), can be
regarded as a lower-order contribution, and the blow-up of \(S\) follows from
a Riccati-type argument.

On the other hand, when \(1<\lambda<2\), we have
\(
  \frac{2}{\lambda}<2.
\)
Thus the \(L^{2/\lambda}\)-estimate for \(R\) and \(S\) is no longer sufficient
to control the integral of \(S^2\) along the characteristic direction of \(R\).
Consequently, it is not clear a priori whether \(R\) remains small compared
with \(S\); indeed, both Riemann variables may grow simultaneously. To overcome
this difficulty, we introduce the supremum functions
\[
  P(t)=\sup_{x\in\mathbb R}S(t,x),
  \qquad
  A(t)=\sup_{x\in\mathbb R}(-R(t,x)),
\]
and first prove that
\[
  A(t)\le P(t)
\]
by a comparison argument based on right Dini derivatives. We then improve this
rough estimate to
\[
  A(t)\le C E_0 P(t)^{2-2/\lambda}.
\]
This refined estimate is sufficient to control \(u\) and hence \(c'(u)/c(u)\),
while still allowing \(A(t)\) to grow as \(P(t)\) grows. Combining these
estimates, we derive a Riccati-type differential inequality for \(P(t)\), which
leads to the finite-time blow-up.

For comparison, in the case \(\lambda=2\), the term \(-(2-\lambda)R^2\) does
not appear in the equation for \(S\). This structural simplification is one of
the reasons why rather precise blow-up criteria are available for the
\(p\)-system. In that case, time-dependent estimates for quantities related to
\(u\), \(c(u)\), and \(c'(u)\) play an important role; see, for example,
\cite{G,GC1,GC2,GPZ,s3,s4}.

Such extremal functions have already been used in the analysis of shallow
water equations to prove wave breaking; see, for example,
\cite{ACJE19982,ACJE1998}. However, to the best of the author's knowledge,
the present paper is the first to derive singularity formation from the
relation between the supremum functions of two unknown variables.

This method is expected to provide a new approach to proving finite-time
blow-up for other equations in which a certain term may suppress the growth
of the component responsible for blow-up.

The remainder of the present paper is organized as follows. In Section 2, we
recall some formulas on characteristic curves and basic properties of the
Riemann variables. We also introduce the supremum functions and the Dini
derivatives used in the proof. In Sections 3 and 4, we prove Theorem
\ref{mainthm} and Proposition \ref{prop:holder}, respectively.

\noindent
{\bf Notation}

We denote  the Lebesgue space for $1\leq p\leq \infty$ and the $L^2$-based Sobolev space of order $m \in \N$ on $\R$ by $L^p (\R)$  and $H^m (\R)$.
For a Banach space $X$, $C^j ([0,T];X)$  denotes the set of functions $f:[0,T] \rightarrow X$ such that $f(t)$ and  its time derivatives up to $j$ are continuous. 
Various positive constants are simply denoted by $C$.

\section{Preliminary}

\subsection{Riemann variables}
We set $R(t,x)$ and $S(t,x)$ as follows.
\begin{eqnarray}\label{ri}\left\{
\begin{array}{ll}
R =\partial_t u +c(u) \partial_x u, \\
S =\partial_t u -c(u) \partial_x u.
\end{array}\right.
\end{eqnarray}
The functions $R$ and $S$ have been used in Glassey, Hunter and Zheng \cite{GHZ, GHZ2} and Zhang and Zheng \cite{zz1}.
We recall some properties of $R$ and $S$ proven in \cite{s1}.  

By \eqref{req}, $R$ and $S$ are  solutions to the following system of first-order equations:
\begin{eqnarray}\label{fs}
 \left\{  
\begin{array}{ll}
\partial_- R =\dfrac{c'(u)}{2c(u)}(R S-S ^2) 
+\lambda  \dfrac{c'(u)}{4c(u)}(R -S )^2 , \\
\partial_x u = \dfrac{1}{2c(u)}(R - S),\\
\partial_+ S  =\dfrac{c'(u)}{2c(u)}(S R-R ^2)
+\lambda  \dfrac{c'(u)}{4c(u)}(S -R )^2 ,
\end{array}
\right.
\end{eqnarray} 
where we define the differential operators $\partial_{\pm}$ by
\[
\partial_{\pm}= \partial_t  \pm c(u)\partial_x .
\]
Let \(x_-(t)\) and \(x_+(t)\) be characteristic curves for the first and the
third equations in \eqref{fs}, respectively.
 In other words, $x_{\pm} (t)$ are solutions to the following differential equations:
 \begin{eqnarray}\label{cc}
\dfrac{d}{dt} x_{\pm} (t)=\pm c(u(t,x_{\pm} (t))).
\end{eqnarray}
We also define the characteristic curves with \(x\) as the independent variable
by
 \begin{eqnarray}\label{ccx}
\dfrac{d}{dx} t_{\pm} (x)=\pm \dfrac{1}{c(u(t_{\pm} (x),x))}.
\end{eqnarray}
Based on their definitions, $R$ and $S$ can be expressed on the characteristic curves $x_{\pm} (t)$ by
\begin{align*}
R(t,x_{+} (t)) = \dfrac{d}{dt} u(t,x_{+} (t)), \\
S(t,x_{-} (t)) = \dfrac{d}{dt} u(t,x_{-} (t)). 
\end{align*}
Next we rewrite  \eqref{fs} on the characteristic curves using this equality. 
Multiplying the first equation in \eqref{fs} by $c^{(\lambda-1)/2}=(c(u))^{(\lambda-1)/2}$ and using the method of characteristics,  we have that the first equation of \eqref{fs} is reduced to
\begin{eqnarray} \label{ch-req}
\dfrac{d}{dt} \left(   c^{(\lambda-1)/2} R (t,x_{-} (t)) \right) =  \dfrac{c' c^{(\lambda-3)/2}}{4} \left(  \lambda R^2  -  (2 - \lambda) S^2 \right) .
\end{eqnarray}
Similarly, it holds for $S$ that
\begin{eqnarray} \label{ch-seq}
\dfrac{d}{dt} \left(   c^{(\lambda-1)/2} S (t,x_{+} (t)) \right) =  \dfrac{c' c^{(\lambda-3)/2}}{4} \left(  \lambda S^2  -  (2 - \lambda) R^2 \right) .
\end{eqnarray}
For $\lambda \in [0,2]$, the equations \eqref{ch-req} and \eqref{ch-seq} immediately imply the following Lemma.
 \begin{lemma}\label{zz1} 
Let $0 \leq \lambda \leq 2$, $c(u)>0$ and  $c'(u) \geq 0$  for $C^1$-solution $u$   of \eqref{req} on $[0, T^*)$. Suppose that $R(0,x) \leq 0$ ($S(0,x) \leq 0$ ). Then  it holds that
\begin{eqnarray}\label{lemmazzes1}
R(t,x) \leq 0 \ (S(t,x) \leq 0 \ \mbox{resp.}) \ \mbox{for} \ (t,x) \in [0,T ^* ) \times \R,
\end{eqnarray}
where $R$ and $S$ are the functions in $(\ref{ri})$.
\end{lemma}

\subsection{Dini derivatives}
We introduce some basic properties of the Dini derivative.

For a continuous function \(f=f(t)\), we define the right upper Dini derivative by
\[
  D_+ f(t)
  =
  \limsup_{h\downarrow 0}
  \frac{f(t+h)-f(t)}{h}.
\]
We also define the right lower Dini derivative by
\[
  D_- f(t)
  =
  \liminf_{h\downarrow 0}
  \frac{f(t+h)-f(t)}{h}.
\]
Then it holds that
\[
  D_+(f-g)(t)
  \le
  D_+f(t)-D_-g(t)
\]
whenever the right-hand side is meaningful.

The Lipschitz continuity of the following supremum function is proved in almost
the same way as in Constantin and Escher \cite{ACJE1998}. For completeness of the present paper, we give the proof.

\begin{lemma}\label{lem:sup_lipschitz}
Let \(w\in C^1([0,T]\times\mathbb R)\). Suppose that
\[
  \sup_{(t,x)\in[0,T]\times\mathbb R}|w(t,x)|+ \sup_{(t,x)\in[0,T]\times\mathbb R}|w_t(t,x)|<\infty
\]
and that
\[
  W(t)=\sup_{x\in\mathbb R}w(t,x)
\]
is finite for all \(t\in[0,T]\). Then \(W\) is Lipschitz continuous on \([0,T]\).
In particular, \(W\) is differentiable for almost every \(t\in[0,T]\).
\end{lemma}

\begin{proof}
Set
\[
  M=\sup_{(t,x)\in[0,T]\times\mathbb R}|w_t(t,x)|.
\]
For \(s,t\in[0,T]\) and \(x\in\mathbb R\), we have
\[
  w(t,x)-w(s,x)
  =
  \int_s^t w_\tau(\tau,x)\,d\tau.
\]
Hence
\[
  |w(t,x)-w(s,x)|
  \le
  M|t-s|.
\]
Therefore
\[
  w(t,x)\le w(s,x)+M|t-s|.
\]
Taking the supremum with respect to \(x\in\mathbb R\), we obtain
\[
  W(t)\le W(s)+M|t-s|.
\]
Interchanging \(s\) and \(t\), we also obtain
\[
  W(s)\le W(t)+M|t-s|.
\]
Thus
\[
  |W(t)-W(s)|\le M|t-s|.
\]
Therefore \(W\) is Lipschitz continuous on \([0,T]\).

By Rademacher's theorem, \(W\) is differentiable for almost every \(t\in[0,T]\).
\end{proof}

We recall some basic properties of the Dini derivative. Some of the proofs are
based on the textbook of  Lakshmikantham and  Leela
\cite{Lakshmikantham1969}.

\begin{lemma}\label{lem:dini_sup}
Let \(w\in C^1([0,T]\times\mathbb R)\) and suppose that
\[
  w_t+b(t,x)w_x=F(t,x)
\]
on \([0,T]\times\mathbb R\), where \(b\) and \(F\) are continuous.
Assume that there exists a compact set \(K\subset\mathbb R\) such that
\[
  \operatorname{supp} w(t,\cdot)\subset K
\]
for all \(t\in[0,T]\). Let
\[
  W(t)=\sup_{x\in\mathbb R}w(t,x)
\]
and
\[
  \mathcal M_w(t)
  =
  \{x\in\mathbb R\mid w(t,x)=W(t)\}.
\]
Then, for \(t\in[0,T)\),
\[
  D_+W(t)
  \le
  \sup_{x\in\mathcal M_w(t)}F(t,x)
\]
and
\[
  D_-W(t)
  \ge
  \inf_{x\in\mathcal M_w(t)}F(t,x).
\]
\end{lemma}

\begin{proof}
Since \(\operatorname{supp} w(t,\cdot)\subset K\) for all \(t\in[0,T]\), the function \(W\) is finite and the supremum is attained. We first prove the estimate for \(D_+W(t)\).

Fix \(t\in[0,T)\). Let \(h_n\downarrow0\) be a sequence such that
\[
  D_+W(t)
  =
  \lim_{n\to\infty}
  \frac{W(t+h_n)-W(t)}{h_n}.
\]
Since \(\operatorname{supp}w(t,\cdot)\subset K\) for all \(t\in[0,T]\), we may choose
\(x_n\in K\) such that
\[
  W(t+h_n)=w(t+h_n,x_n).
\]
Taking a subsequence if necessary, we may assume that
\[
  x_n\to x_0\in K.
\]
By Lemma \ref{lem:sup_lipschitz}, \(W\) is continuous. Hence
\[
  W(t+h_n)\to W(t).
\]
On the other hand, since \(W(t+h_n)=w(t+h_n,x_n)\), \(h_n\downarrow0\),
and \(x_n\to x_0\), the continuity of \(w\) gives
\[
  W(t)
  =\lim_{n\to\infty}W(t+h_n)
  =\lim_{n\to\infty}w(t+h_n,x_n)
  =w(t,x_0).
\]
Thus \(x_0\in\mathcal M_w(t)\).

Since
\[
  w(t,x_n)\le W(t),
\]
we obtain
\[
  \frac{W(t+h_n)-W(t)}{h_n}
  \le
  \frac{w(t+h_n,x_n)-w(t,x_n)}{h_n}.
\]
Moreover,
\[
  \frac{w(t+h_n,x_n)-w(t,x_n)}{h_n}
  =
  \int_0^1 w_t(t+\theta h_n,x_n)\,d\theta.
\]
Since \(w_t\) is continuous on the compact set \([0,T]\times K\), it follows that
\[
  \int_0^1 w_t(t+\theta h_n,x_n)\,d\theta
  \to
  w_t(t,x_0).
\]
Therefore
\[
  D_+W(t)\le w_t(t,x_0).
\]
By the equation for \(w\),
\[
  w_t(t,x_0)=F(t,x_0)-b(t,x_0)w_x(t,x_0).
\]
Since \(x_0\) is a maximum point of \(w(t,\cdot)\), it follows that
\[
  w_x(t,x_0)=0.
\]
Therefore
\[
  D_+W(t)
  \le
  F(t,x_0)
  \le
  \sup_{x\in\mathcal M_w(t)}F(t,x).
\]

Next we prove the estimate for \(D_-W(t)\). Let \(x_0\in\mathcal M_w(t)\). Let \(X(s)\) be the solution to
\[
  \frac{d}{ds}X(s)=b(s,X(s)),
  \qquad
  X(t)=x_0.
\]
Then
\[
  W(t+h)\ge w(t+h,X(t+h)).
\]
Since \(W(t)=w(t,x_0)\), we have
\[
  \frac{W(t+h)-W(t)}{h}
  \ge
  \frac{w(t+h,X(t+h))-w(t,x_0)}{h}.
\]
Moreover,
\[
  \frac{d}{ds}w(s,X(s))=F(s,X(s)).
\]
Hence
\[
  \frac{w(t+h,X(t+h))-w(t,x_0)}{h}
  =
  \frac1h\int_t^{t+h}F(s,X(s))\,ds.
\]
Letting \(h\downarrow0\), we obtain
\[
  D_-W(t)\ge F(t,x_0).
\]
Since \(x_0\in\mathcal M_w(t)\) is arbitrary, this implies in particular
\[
  D_-W(t)
  \ge
  \inf_{x\in\mathcal M_w(t)}F(t,x).
\]
The proof is complete.
\end{proof}

\section{Proof of Theorem \ref{mainthm}}

\subsection{Bootstrap argument}
We shall use the standard continuation criterion for classical solutions:
as long as \(u\) remains in a compact subset of an interval on which \(c\)
is smooth and positive, and
\[
  \|u_t(t)\|_{L^\infty}+\|u_x(t)\|_{L^\infty}
\]
remains bounded, the classical solution can be continued.

Let
\[
  1<\lambda<2,
  \qquad
  p=\frac{2}{\lambda},
  \qquad
  \theta=2-p.
\]
Then
\[
  1<p<2,
  \qquad
  0<\theta<1.
\]
We set
\[
  r=-R.
\]
On any time interval on which \(c'(u)\ge0\), Lemma 2.1 gives
\(R(t,x)\le0\). In particular, this holds on the bootstrap interval
because of (3.3).

We introduce
\[
  P(t)=\sup_{x\in\mathbb R}S(t,x),
  \qquad
  A(t)=\sup_{x\in\mathbb R}r(t,x),
  \qquad
  L(t)=\sup_{x\in\mathbb R}(-S(t,x)).
\]
We also set
\[
  E(t)
  =
  \int_{\mathbb R}
  \left(
    |R(t,x)|^p+|S(t,x)|^p
  \right)\,dx,
  \qquad
  E_0=E(0).
\]
For the initial data in Theorem \ref{mainthm}, we have
\[
  R(0,x)=0
\]
and
\[
  S(0,x)
  =
  -2c(\delta\phi(x/\varepsilon))
  \frac{\delta}{\varepsilon}
  \phi'(x/\varepsilon).
\]
Therefore
\[
  E_0
  =
  \int_{\mathbb R}|S(0,x)|^p\,dx
  \le
  C\left(\frac{\delta}{\varepsilon}\right)^p\varepsilon .
\]
Moreover, by the assumption on \(\phi\), if \(\delta>0\) is sufficiently small, then
\[
  P(0)>0,
  \qquad
  A(0)=0<P(0),
  \qquad
  L(0)<P(0).
\]

Let \(K_1,K_2\) be positive constants which will be fixed later.
We consider a time interval \([0,T]\subset[0,T^*)\) on which the following estimates hold.

First, we assume
\begin{equation}\label{bt-u}
  \sup_{0\le t\le T}\|u(t)\|_{L^\infty}
  \le
  K_1\delta .
\end{equation}
Choosing \(\delta>0\) sufficiently small, \eqref{bt-u} implies that
\begin{equation}\label{bt-c}
  \frac{c_0}{2}
  \le
  c(u(t,x))
  \le
  2c_0
\end{equation}
and
\begin{equation}\label{bt-a}
  0<a_*
  \le
  \frac{c'(u(t,x))}{4c(u(t,x))}
  \le
  a^*
\end{equation}
for \((t,x)\in[0,T]\times\mathbb R\), where \(a_*\) and \(a^*\) are positive constants independent of \(\delta\) and \(\varepsilon\).

Second, we assume the \(L^p\)-estimate  
\begin{equation}\label{bt-E}
  E(t)\le 4E_0
\end{equation}
for \(0\le t\le T\).

Third, for every minus characteristic curve \(x_-(s)\), we assume that 
\begin{equation}\label{bt-char}
  \int_0^t |S(s,x_-(s))|^p\,ds
  \le
  \cfrac{8}{c_0} E_0
\end{equation}
for \(0\le t\le T\).

Finally, we assume the refined estimate for \(r=-R\)
\begin{equation}\label{bt-A}
  A(t)\le K_2 E_0P(t)^\theta
\end{equation}
for \(0\le t\le T\).

Under the assumption \eqref{bt-u}, we shall first prove that
\begin{equation}\label{bt-cone-pre}
  A(t) \le P(t),
  \qquad
  L(t) \le P(t)
\end{equation}
for \(0\le t\le T\). The inequalities \eqref{bt-cone-pre} are not included in the bootstrap assumptions. They follow from the initial strict inequalities
\[
  A(0)<P(0),
  \qquad
  L(0)<P(0),
\]
and the positivity of \(c'(u)/c(u)\), which is guaranteed by \eqref{bt-u}. The estimate \eqref{bt-cone-pre} will be used to derive the Riccati-type lower estimate for \(P(t)\).

Our goal is to improve \eqref{bt-u}, \eqref{bt-E}, \eqref{bt-char}, and \eqref{bt-A}, provided \(\delta>0\) is sufficiently small. Since the improved estimates are independent of \(T<T^*\), the standard continuity argument shows that these estimates hold on every compact subinterval of \([0,T^*)\). The Riccati-type lower estimate for \(P(t)\) then implies that \(T^*\) is finite.

We first prove the following proposition. This proposition is not itself one of the bootstrap estimates, but an inequality between the supremum functions
which follows under the bootstrap assumptions.

\begin{proposition}\label{prop:riccati_P}
For the solution corresponding to the initial data of Theorem \ref{mainthm}, assume that \eqref{bt-c} and \eqref{bt-a} hold on \([0,T]\). Suppose that
\[
  A(0)<P(0),
  \qquad
  L(0)<P(0).
\]
Then
\begin{equation}\label{bt-cone}
  A(t)\le P(t),
  \qquad
  L(t)\le P(t)
\end{equation}
for \(0\le t\le T\). Moreover,
\begin{equation}\label{P-riccati-dini}
  D_-P(t)\ge a_*b_\lambda P(t)^2
\end{equation}
for \(0\le t<T\), where
\[
  b_\lambda=\min\{\lambda,4(\lambda-1)\}>0.
\]
Consequently, for \(0\le t \leq T\),
\begin{equation}\label{P-lower}
  P(t)\ge
  \frac{P(0)}{1-a_*b_\lambda P(0)t}
\end{equation}
as long as the right-hand side is finite. Moreover, for every \(0<\gamma<1\) and \(0\le t \leq T\),
\begin{equation}\label{P-gamma-int}
  \int_0^t P(s)^\gamma\,ds
  \le
  \frac{1}{a_*b_\lambda(1-\gamma)}
  P(0)^{\gamma-1}.
\end{equation}
\end{proposition}

\begin{proof}
We set
\[
  r=-R.
\]
By Lemma \ref{zz1}, we have \(R(t,x)\le0\), and hence
\[
  r(t,x)\ge0.
\]

Since the initial data are compactly supported and \(c(u)\) is bounded on
\([0,T]\times\mathbb R\), the finite propagation speed implies that \(R\)
and \(S\) have supports contained in a fixed compact interval on \([0,T]\).
Thus Lemma \ref{lem:dini_sup} is applicable to \(S\), \(r\), and \(-S\).

We first prove \eqref{bt-cone}. Assume, contrary to the assertion, that one of the inequalities in \eqref{bt-cone} fails. Since
\[
  A(0)<P(0),
  \qquad
  L(0)<P(0),
\]
there exists a first time \(t_0\in(0,T)\) such that
\[
  A(t)\le P(t),
  \qquad
  L(t)\le P(t)
\]
for \(0\le t\le t_0\), and either \(A-P\) or \(L-P\) becomes positive immediately after \(t_0\). At \(t=t_0\), at least one of the following equalities holds:
\[
  A(t_0)=P(t_0),
  \qquad
  L(t_0)=P(t_0).
\]

On \([0,t_0]\), we have
\[
  A(t)\le P(t),
  \qquad
  L(t)\le P(t).
\]
Let \(x\) be a maximum point of \(S(t,\cdot)\). Then
\[
  S(t,x)=P(t),
  \qquad
  0\le r(t,x)\le A(t)\le P(t).
\]
The equation for \(S\) is
\[
  \partial_+S
  =
  \frac{c'(u)}{4c(u)}
  \left\{
    \lambda S^2
    +2(\lambda-1)rS
    -(2-\lambda)r^2
  \right\}.
\]
For fixed \(P(t)\ge0\), the function
\[
  z\mapsto
  \lambda P(t)^2
  +2(\lambda-1)zP(t)
  -(2-\lambda)z^2
\]
is concave on \(0\le z\le P(t)\). Hence its minimum on this interval is attained at an endpoint. Therefore, for \(0\le z\le P(t)\),
\[
  \lambda P(t)^2
  +2(\lambda-1)zP(t)
  -(2-\lambda)z^2
  \ge
  b_\lambda P(t)^2,
\]
where
\[
  b_\lambda=\min\{\lambda,4(\lambda-1)\}.
\]
Using \eqref{bt-a}, we obtain
\[
  \partial_+S(t,x)\ge a_*b_\lambda P(t)^2
\]
at every maximum point of \(S(t,\cdot)\). 
By Lemma \ref{lem:dini_sup},
\[
  D_-P(t)\ge a_*b_\lambda P(t)^2
\]
for \(0\le t\le t_0\). In particular, \(P(t)>0\) on \([0,t_0]\).

We next exclude the possibility that \(A(t_0)=P(t_0)\) and that \(A-P\) becomes positive immediately after \(t_0\). Let \(x\) be a maximum point of \(r(t_0,\cdot)\). Then
\[
  r(t_0,x)=A(t_0)=P(t_0).
\]
Since \(L(t_0)\le P(t_0)\), we have
\[
  -P(t_0)\le S(t_0,x)\le P(t_0).
\]
The equation for \(r\) is
\[
  \partial_-r
  =
  \frac{c'(u)}{4c(u)}
  \left\{
    (2-\lambda)S^2
    -2(\lambda-1)rS
    -\lambda r^2
  \right\}.
\]
For every \(-P(t_0)\le \sigma\le P(t_0)\), we have
\[
  (2-\lambda)\sigma^2
  -2(\lambda-1)P(t_0)\sigma
  -\lambda P(t_0)^2
  \le0.
\]
Indeed, the left-hand side is a convex function of \(\sigma\), and at the endpoints it equals \(0\) and \(4(1-\lambda)P(t_0)^2\), respectively. Hence
\[
  \partial_-r(t_0,x)\le0
\]
at every maximum point of \(r(t_0,\cdot)\). By Lemma \ref{lem:dini_sup},
\[
  D_+A(t_0)\le0.
\]
On the other hand,
\[
  D_-P(t_0)\ge a_*b_\lambda P(t_0)^2>0.
\]
Therefore
\[
  D_+(A-P)(t_0)
  \le
  D_+A(t_0)-D_-P(t_0)
  <0.
\]
This is impossible. Indeed, if \(A-P\) becomes positive immediately after \(t_0\), then there exists a sequence \(h_n\downarrow0\) such that
\[
  (A-P)(t_0+h_n)>0.
\]
Since
\[
  (A-P)(t_0)=0,
\]
we have
\[
  D_+(A-P)(t_0)
  \ge
  \limsup_{n\to\infty}
  \frac{(A-P)(t_0+h_n)-(A-P)(t_0)}{h_n}
  \ge 0.
\]
This contradicts
\[
  D_+(A-P)(t_0)<0.
\]
Therefore \(A-P\) cannot become positive immediately after \(t_0\).

It remains to exclude the possibility that \(L(t_0)=P(t_0)\) and that \(L-P\) becomes positive immediately after \(t_0\). Set
\[
  \ell=-S.
\]
Then
\[
  L(t)=\sup_{x\in\mathbb R}\ell(t,x).
\]
The equation for \(\ell\) is
\[
  \partial_+\ell
  =
  \frac{c'(u)}{4c(u)}
  \left\{
    -\lambda \ell^2
    +2(\lambda-1)r\ell
    +(2-\lambda)r^2
  \right\}.
\]
Let \(x\) be a maximum point of \(\ell(t_0,\cdot)\). Then
\[
  \ell(t_0,x)=L(t_0)=P(t_0).
\]
Moreover,
\[
  0\le r(t_0,x)\le A(t_0)\le P(t_0).
\]
For every \(0\le \rho\le P(t_0)\), we have
\[
  -\lambda P(t_0)^2
  +2(\lambda-1)\rho P(t_0)
  +(2-\lambda)\rho^2
  \le0.
\]
Indeed, the left-hand side is a convex function of \(\rho\), and its values at \(\rho=0\) and \(\rho=P(t_0)\) are \(-\lambda P(t_0)^2\) and \(0\), respectively. Hence
\[
  \partial_+\ell(t_0,x)\le0
\]
at every maximum point of \(\ell(t_0,\cdot)\). By Lemma \ref{lem:dini_sup},
\[
  D_+L(t_0)\le0.
\]
Since
\[
  D_-P(t_0)\ge a_*b_\lambda P(t_0)^2>0,
\]
we get
\[
  D_+(L-P)(t_0)
  \le
  D_+L(t_0)-D_-P(t_0)
  <0.
\]
This contradicts the fact that \(L-P\) becomes positive immediately after \(t_0\).

Thus \eqref{bt-cone} holds on \([0,T]\). The argument above also gives
\[
  D_-P(t)\ge a_*b_\lambda P(t)^2
\]
for \(0\le t<T\). This proves \eqref{P-riccati-dini}.

By Lemma \ref{lem:sup_lipschitz}, \(P\) is locally Lipschitz. Hence, at
almost every point of differentiability of \(P\), we have
\(D_-P(t)=P'(t)\). Thus \eqref{P-riccati-dini} gives
\[
  P'(t)\ge a_*b_\lambda P(t)^2
\]
for almost every \(t\in[0,T)\). Hence
\[
  \frac{d}{dt}\left(\frac1{P(t)}\right)
  \le
  -a_*b_\lambda
\]
for almost every \(t\in[0,T)\). Integrating over \([0,t]\), where \(0\le t<T\), we obtain
\[
  \frac1{P(t)}
  \le
  \frac1{P(0)}-a_*b_\lambda t.
\]
This gives \eqref{P-lower}.

Finally, let \(0<\gamma<1\). Since
\[
  P'(t)\ge a_*b_\lambda P(t)^2
\]
for almost every \(t\in[0,T)\), we have
\[
  \frac{d}{dt}P(t)^{\gamma-1}
  \le
  -a_*b_\lambda(1-\gamma)P(t)^\gamma
\]
for almost every \(t\in[0,T)\). Integrating this inequality over \([0,t]\), where \(0\le t<T\), we obtain
\[
  a_*b_\lambda(1-\gamma)
  \int_0^t P(s)^\gamma\,ds
  \le
  P(0)^{\gamma-1}-P(t)^{\gamma-1}.
\]
Since \(P(t)>0\),
\[
  P(0)^{\gamma-1}-P(t)^{\gamma-1}
  \le
  P(0)^{\gamma-1}.
\]
Therefore \eqref{P-gamma-int} follows.
\end{proof}

Next, we improve the estimate for \(E(t)\).

\begin{proposition}\label{prop:E_improve}
For the solution corresponding to the initial data of Theorem \ref{mainthm}, assume that \eqref{bt-c}, \eqref{bt-a} and \eqref{bt-A} hold on \([0,T]\).
Then there exists \(\delta_1>0\), depending on \(c\), \(\lambda\), \(\phi\), and the constant in \eqref{bt-A}, but independent of \(\varepsilon\) and \(T\), such that if
\(0<\delta\le \delta_1\), then
\[
  E(t)\le 2E_0
\]
for \(0\le t\le T\).
\end{proposition}

\begin{proof}
By Proposition \ref{prop:riccati_P}, we have \eqref{bt-cone} and \eqref{P-gamma-int} on \([0,T]\).
We recall that
\[
  p=\frac{2}{\lambda},
  \qquad
  \theta=2-p.
\]
We first derive the \(L^p\)-identity. Multiplying the first equation of
\eqref{fs} by \(|R|^{p-2}R\), we obtain
\begin{equation}\label{lp-R-identity}
\begin{aligned}
  \frac1p
  \left\{
    \partial_t |R|^p
    -
    \partial_x\left(c(u)|R|^p\right)
  \right\}
  &=
  -\frac{c'(u)}{2pc(u)}(R-S)|R|^p  \\
  &\quad
  +\frac{c'(u)}{2c(u)}(RS-S^2)|R|^{p-2}R  \\
  &\quad
  +\lambda\frac{c'(u)}{4c(u)}(R-S)^2|R|^{p-2}R .
\end{aligned}
\end{equation}
Similarly, multiplying the third equation of \eqref{fs} by
\(|S|^{p-2}S\), we have
\begin{equation}\label{lp-S-identity}
\begin{aligned}
  \frac1p
  \left\{
    \partial_t |S|^p
    +
    \partial_x\left(c(u)|S|^p\right)
  \right\}
  &=
  \frac{c'(u)}{2pc(u)}(R-S)|S|^p  \\
  &\quad
  +\frac{c'(u)}{2c(u)}(SR-R^2)|S|^{p-2}S  \\
  &\quad
  +\lambda\frac{c'(u)}{4c(u)}(S-R)^2|S|^{p-2}S .
\end{aligned}
\end{equation}
Here we used
\[
  \partial_x c(u)
  =
  c'(u)\partial_xu
  =
  \frac{c'(u)}{2c(u)}(R-S).
\]
Adding the two identities \eqref{lp-R-identity} and \eqref{lp-S-identity}, we get
\[
\begin{aligned}
  &\frac1p
  \left\{
    \partial_t(|R|^p+|S|^p)
    -
    \partial_x\left(c(u)(|R|^p-|S|^p)\right)
  \right\} \\
  &=
  \frac{c'(u)}{c(u)}
  \left[
    -\frac{1}{2p}(R-S)(|R|^p-|S|^p)
  \right.\\
  &\qquad\left.
    +\frac12(R-S)
    \left\{
      S|R|^{p-2}R
      -
      R|S|^{p-2}S
    \right\}
  \right.\\
  &\qquad\left.
    +\frac{\lambda}{4}(R-S)^2
    \left(
      |R|^{p-2}R+|S|^{p-2}S
    \right)
  \right].
\end{aligned}
\]
Since \(p=2/\lambda\), we have
\[
  \frac{\lambda}{4}=\frac{1}{2p}.
\]
Moreover,
\[
\begin{aligned}
  &(R-S)(|R|^p-|S|^p)
  -(R-S)^2
  \left(
    |R|^{p-2}R+|S|^{p-2}S
  \right)  \\
  &\qquad =
  (R-S)
  \left\{
    S|R|^{p-2}R
    -
    R|S|^{p-2}S
  \right\}.
\end{aligned}
\]
Therefore, we obtain
\begin{equation}\label{lp-identity}
\begin{aligned}
  &\frac1p
  \left\{
    \partial_t(|R|^p+|S|^p)
    -
    \partial_x\left(c(u)(|R|^p-|S|^p)\right)
  \right\}  \\
  &\qquad =
  \left(
    \frac12-\frac{\lambda}{4}
  \right)
  \frac{c'(u)}{c(u)}
  (R-S)
  \left\{
    S|R|^{p-2}R
    -
    R|S|^{p-2}S
  \right\}.
\end{aligned}
\end{equation}
We next estimate the right-hand side of \eqref{lp-identity}. Since \(R\le0\), we put
\[
  r=-R.
\]
Then \(|R|=r\). By \eqref{bt-cone}, we have
\[
  0\le r(t,x)\le A(t),
  \qquad
  |S(t,x)|\le P(t).
\]
We claim that
\begin{equation*}
\begin{aligned}
  \left|
    (R-S)
    \left\{
      S|R|^{p-2}R
      -
      R|S|^{p-2}S
    \right\}
  \right|
  \le
  4
  \left\{
    A(t)r^p
    +
    A(t)^{p-1}P(t)^{2-p}|S|^p
  \right\}.
\end{aligned}
\end{equation*}
Indeed, the left-hand side is bounded by
\[
  (r+|S|)
  \left(
    |S|r^{p-1}
    +
    r|S|^{p-1}
  \right).
\]
If \(|S|\le r\), then
\[
  (r+|S|)
  \left(
    |S|r^{p-1}
    +
    r|S|^{p-1}
  \right)
  \le
  4 r^{p+1}
  \le
  4 A(t)r^p.
\]
If \(r\le |S|\), then
\[
\begin{aligned}
  (r+|S|)
  \left(
    |S|r^{p-1}
    +
    r|S|^{p-1}
  \right)
  &\le
  2
  \left(
    r^{p-1}|S|^2
    +
    r|S|^p
  \right)  \\
  &\le
  4 A(t)^{p-1}P(t)^{2-p}|S|^p.
\end{aligned}
\]
Thus we obtain the above inequality.

By \eqref{bt-A}, we have
\[
  A(t)\le K_2E_0P(t)^\theta.
\]
Since \(2-p=\theta\), it follows that
\[
  A(t)^{p-1}P(t)^{2-p}
  \le
  C E_0^{p-1}P(t)^{p\theta}.
\]
Therefore
\[
  \frac{d}{dt}E(t)
  \le
  C
  \left\{
    E_0P(t)^\theta
    +
    E_0^{p-1}P(t)^{p\theta}
  \right\}
  E(t).
\]
Since
\[
  0<\theta<1
\]
and
\[
  p\theta=p(2-p)<1,
\]
we may apply \eqref{P-gamma-int} with \(\gamma=\theta\) and \(\gamma=p\theta\). Hence
\[
\begin{aligned}
  \int_0^t
  \left\{
    E_0P(s)^\theta
    +
    E_0^{p-1}P(s)^{p\theta}
  \right\}\,ds
  \le
  C
  \left\{
    E_0P(0)^{\theta-1}
    +
    E_0^{p-1}P(0)^{p\theta-1}
  \right\}.
\end{aligned}
\]
For the initial data in Theorem \ref{mainthm}, we have
\[
  S(0,x)
  =
  -2c(\delta\phi(x/\varepsilon))
  \frac{\delta}{\varepsilon}\phi'(x/\varepsilon).
\]
Let
\[
  m_\phi=\sup_{x\in\mathbb R}(-\phi'(x))>0.
\]
Taking \(\delta>0\) sufficiently small, we may assume that
\[
  \frac{c_0}{2}\le c(\delta\phi(x))\le 2c_0.
\]
Hence
\[
  P(0)\ge c_0m_\phi\frac{\delta}{\varepsilon}.
\]
Moreover, since \(R(0,x)=0\), by the change of variables \(y=x/\varepsilon\),
\[
\begin{aligned}
  E_0
  &=
  \int_{\mathbb R}|S(0,x)|^p\,dx  \\
  &\le
  (4c_0)^p
  \left(\frac{\delta}{\varepsilon}\right)^p
  \int_{\mathbb R}|\phi'(x/\varepsilon)|^p\,dx  \\
  &=
  (4c_0)^p
  \|\phi'\|_{L^p}^p
  \left(\frac{\delta}{\varepsilon}\right)^p\varepsilon.
\end{aligned}
\]
Therefore,
\[
  E_0P(0)^{\theta-1}\le C\delta,
  \qquad
  E_0^{p-1}P(0)^{p\theta-1}\le C\delta^{p-1}.
\]
Consequently,
\[
  \int_0^t
  \left\{
    E_0P(s)^\theta
    +
    E_0^{p-1}P(s)^{p\theta}
  \right\}\,ds
  \le
  C(\delta+\delta^{p-1}).
\]
By Gronwall's inequality,
\[
  E(t)
  \le
  E_0
  \exp\{C(\delta+\delta^{p-1})\}.
\]
We note that the constants above may depend on $\phi, c, \lambda$.

Choose \(\delta_1>0\) so small that
\[
  \exp\{C(\delta+\delta^{p-1})\}\le 2
\]
for \(0<\delta\le\delta_1\). Then, for such \(\delta\), Gronwall's inequality gives
\[
  E(t)\le 2E_0
\]
for \(0\le t \leq T\). Since \(E(t)\) is continuous in \(t\), the same estimate also holds at \(t=T\).
\end{proof}

The next estimate is an improvement of the energy estimate along characteristic
curves. It can be proved in almost the same way as the improvement of the
estimate for \(E(t)\). Although the same estimate also holds for \(R\), we only
state and prove the estimate for \(S\), which is needed below.

\begin{proposition}\label{prop:char_improve}
For the solution corresponding to the initial data of Theorem \ref{mainthm}, assume that \eqref{bt-c}, \eqref{bt-a}, and \eqref{bt-A} hold on \([0,T]\).
Then there exists \(0<\delta_2\le\delta_1\), depending on \(c\), \(\lambda\), \(\phi\), and \(K_2\), but independent of \(\varepsilon\) and \(T\), such that if \(0<\delta\le\delta_2\), then
\[
  \int_0^t |S(s,x_-(s))|^p\,ds
  \le
  \frac{4}{c_0}E_0
\]
for \(0\le t\le T\) and for every minus characteristic curve \(x_-(s)\).
\end{proposition}

\begin{proof}
By Proposition \ref{prop:riccati_P}, we have \eqref{bt-cone} and \eqref{P-gamma-int} on \([0,T]\).
Let \((t,x)\in [0,T]\times \mathbb R\) be fixed. We denote by
\(t_+(y)\) and \(t_-(y)\) the characteristic curves through \((t,x)\). 
Let \(x_1\) and \(x_2\) be the points where these two characteristics meet
the \(x\)-axis. Namely, \(t_+(x_1)=0\) and
\(t_-(x_2)=0\). We denote by \(\Delta\) the characteristic triangle
surrounded by the two characteristic curves and the interval \([x_1,x_2]\)
on the \(x\)-axis.

Integrating \eqref{lp-identity} over \(\Delta\), we obtain
\[
\begin{aligned}
  &\int_{x_1}^{x}|R(t_+(y),y)|^p\,dy
  +
  \int_x^{x_2}|S(t_-(y),y)|^p\,dy  \\
  &\quad\le
  \frac12
  \int_{x_1}^{x_2}
  \left(
    |R(0,y)|^p+|S(0,y)|^p
  \right)\,dy  \\
  &\qquad
  +
  C\iint_{\Delta}
  \left|
    (R-S)
    \left\{
      S|R|^{p-2}R
      -
      R|S|^{p-2}S
    \right\}
  \right|\,dy\,ds .
\end{aligned}
\]
The first term on the right-hand side is estimated as 
\[
\frac12
  \int_{x_1}^{x_2}
  \left(
    |R(0,y)|^p+|S(0,y)|^p
  \right)\,dy  \leq  \frac12E_0.
\]

In the same way as in Proposition \ref{prop:E_improve}, the double integral
is estimated as
\[
\begin{aligned}
  &\iint_{\Delta}
  \left|
    (R-S)
    \left\{
      S|R|^{p-2}R
      -
      R|S|^{p-2}S
    \right\}
  \right|\,dy\,ds \\
  &\quad\le
  C\int_0^t
  \left\{
    A(s)
    +
    A(s)^{p-1}P(s)^{2-p}
  \right\}
  E(s)\,ds .
\end{aligned}
\]
By Proposition \ref{prop:E_improve}, we have
\[
  E(s)\le 2E_0 .
\]
Moreover, by \eqref{bt-A},
\[
  A(s)\le C E_0P(s)^\theta.
\]
Since \(\theta=2-p\), we also have
\[
  A(s)^{p-1}P(s)^{2-p}
  \le
  C E_0^{p-1}P(s)^{p\theta}.
\]
Therefore,
\[
\begin{aligned}
  &\iint_{\Delta}
  \left|
    (R-S)
    \left\{
      S|R|^{p-2}R
      -
      R|S|^{p-2}S
    \right\}
  \right|\,dy\,ds \\
  &\quad\le
  C E_0
  \int_0^t
  \left\{
    E_0P(s)^\theta
    +
    E_0^{p-1}P(s)^{p\theta}
  \right\}\,ds .
\end{aligned}
\]
Using \eqref{P-gamma-int} with \(\gamma=\theta\) and \(\gamma=p\theta\), and using the same estimates of \(E_0\) and \(P(0)\) in the proof of Proposition \ref{prop:E_improve}, we obtain, after increasing \(C\) if necessary,
\[
\begin{aligned}
  &\iint_\Delta
  \left|
    (R-S)\left\{
      S|R|^{p-2}R
      -
      R|S|^{p-2}S
    \right\}
  \right|\,dy\,ds  \\
  &\qquad\le
  CE_0(\delta+\delta^{p-1}).
\end{aligned}
\]
Choose \(0<\delta_2\le\delta_1\) so small that
\[
  C(\delta+\delta^{p-1})\le \frac32
\]
for \(0<\delta\le\delta_2\). Then
\[
\iint_\Delta
  \left|
    (R-S)\left\{
      S|R|^{p-2}R
      -
      R|S|^{p-2}S
    \right\}
  \right|\,dy\,ds
\le
\frac32 E_0.
\]
Consequently,
\[
\begin{aligned}
  &\int_{x_1}^{x}|R(t_+(y),y)|^p\,dy
  +
  \int_x^{x_2}|S(t_-(y),y)|^p\,dy  \\
  &\quad\le
  \frac12E_0+\frac32E_0
  =
  2E_0 .
\end{aligned}
\]
In particular,
\[
  \int_x^{x_2}|S(t_-(y),y)|^p\,dy
  \le
  2E_0 .
\]

Since
\[
  \frac{d}{dy}t_-(y)
  =
  -\frac{1}{c(u(t_-(y),y))}
\]
and \eqref{bt-c} gives
\[
  c(u(t_-(y),y))\ge \frac{c_0}{2},
\]
we have
\[
  \left|\frac{d}{dy}t_-(y)\right|
  \le
  \frac{2}{c_0}.
\]
Therefore,
\[
\begin{aligned}
  \int_0^t |S(s,x_-(s))|^p\,ds
  &=
  \int_x^{x_2}
  |S(t_-(y),y)|^p
  \left|\frac{d}{dy}t_-(y)\right|\,dy  \\
  &\le
  \frac{2}{c_0}
  \int_x^{x_2}|S(t_-(y),y)|^p\,dy  \\
  &\le
  \frac{4}{c_0}E_0 .
\end{aligned}
\]
This completes the proof.
\end{proof}

\begin{remark}
Let \((t,x)\in [0,T]\times\mathbb R\) be fixed, and let
\(t_+(y)\), \(t_-(y)\), \(x_1\), and \(x_2\) be as in the proof of
Proposition \ref{prop:char_improve}. By the same argument as in
Proposition \ref{prop:char_improve}, we can obtain
\[
  \int_0^t |R(s,x_+(s))|^p\,ds
  \le \frac{4}{c_0}E_0
\tag{3.15}
\]
and
\[
  \int_{x_1}^{x}|R(t_+(y),y)|^p\,dy
  +
  \int_x^{x_2}|S(t_-(y),y)|^p\,dy
  \le C E_0,
\tag{3.16}
\]
where \(C\) is a positive constant.
\end{remark}

The next estimate is an improvement of the estimate for \(A(t)\).

\begin{proposition}\label{prop:A_improve}
For the solution corresponding to the initial data of Theorem \ref{mainthm}, assume that \eqref{bt-c}, \eqref{bt-a}, and \eqref{bt-A}
hold on \([0,T]\). Assume also that \(0<\delta\le \delta_2\), where
\(\delta_2\) is given in Proposition \ref{prop:char_improve}.
Suppose that \(K_2\) is chosen so large that
\[
  K_2
  \ge
  \frac{8(2-\lambda)a^*4^{(\lambda-1)/2}}{c_0}.
\]
Then it holds that
\[
A(t)\le \frac{K_2}{2}E_0P(t)^\theta
\]
for \(0\le t\le T\).
\end{proposition}

\begin{proof}
By Proposition \ref{prop:riccati_P}, we have \eqref{bt-cone} and \eqref{P-gamma-int} on \([0,T]\).
Let \((t,x)\in[0,T]\times\mathbb R\) be fixed. We consider the minus characteristic curve \(x_-(s)\) satisfying
\[
  x_-(t)=x.
\]
We put
\[
  k=\frac{\lambda-1}{2}.
\]
Since \(R(0,x)=0\), it follows from \eqref{ch-req} that
\[
\begin{aligned}
  c(u(t,x))^k R(t,x)
  &=
  \int_0^t
  \frac{c'(u)c(u)^{(\lambda-3)/2}}{4}
  \left(
    \lambda R^2-(2-\lambda)S^2
  \right)(s,x_-(s))\,ds .
\end{aligned}
\]
Recalling that \(r=-R\), we obtain
\[
\begin{aligned}
  c(u(t,x))^k r(t,x)
  &=
  \int_0^t
  \frac{c'(u)c(u)^{(\lambda-3)/2}}{4}
  \left(
    (2-\lambda)S^2-\lambda R^2
  \right)(s,x_-(s))\,ds  \\
  &\le
  (2-\lambda)
  \int_0^t
  \frac{c'(u)}{4c(u)}
  c(u)^k
  S(s,x_-(s))^2\,ds .
\end{aligned}
\]
By \eqref{bt-a} and \eqref{bt-c}, we have
\[
  \frac{c'(u)}{4c(u)}\le a^*
\]
and
\[
  c(u)^k\le (2c_0)^k,
  \qquad
  c(u(t,x))^k\ge \left(\frac{c_0}{2}\right)^k.
\]
Therefore,
\[
  r(t,x)
  \le
  (2-\lambda)a^*4^k
  \int_0^t |S(s,x_-(s))|^2\,ds .
\]

By Proposition \ref{prop:riccati_P}, \(P(t)\) is nondecreasing on \([0,T]\). Hence, since \(p<2\),
\[
\begin{aligned}
  \int_0^t |S(s,x_-(s))|^2\,ds
  &\le
  P(t)^{2-p}
  \int_0^t |S(s,x_-(s))|^p\,ds .
\end{aligned}
\]
Using Proposition \ref{prop:char_improve}, we get
\[
  \int_0^t |S(s,x_-(s))|^p\,ds
  \le
  \frac{4}{c_0}E_0 .
\]
Thus
\[
  r(t,x)
  \le
  \frac{4(2-\lambda)a^*4^k}{c_0}
  E_0P(t)^{2-p}.
\]
Since
\[
  \theta=2-p,
\]
we have
\[
  r(t,x)
  \le
  \frac{4(2-\lambda)a^*4^{(\lambda-1)/2}}{c_0}
  E_0P(t)^\theta.
\]
Taking the supremum with respect to \(x\in\mathbb R\), we obtain
\[
  A(t)
  \le
  \frac{4(2-\lambda)a^*4^{(\lambda-1)/2}}{c_0}
  E_0P(t)^\theta.
\]
By the choice of \(K_2\),
\[
  \frac{4(2-\lambda)a^*4^{(\lambda-1)/2}}{c_0}
  \le
  \frac{K_2}{2}.
\]
Therefore,
\[
  A(t)
  \le
  \frac{K_2}{2}E_0P(t)^\theta.
\]
This completes the proof.
\end{proof}

Finally, we improve the estimate for \(u\). This also improves the upper and
lower bounds for \(c(u)\).

\begin{proposition}\label{prop:u_improve}
For the solution corresponding to the initial data of Theorem \ref{mainthm}, suppose that the conclusion of Proposition \ref{prop:A_improve} holds on
\([0,T]\). Let
\[
  m_\phi=\sup_{x\in\mathbb R}(-\phi'(x)).
\]
By \(\phi'(0)<0\), we have \(m_\phi>0\).
Set
\[
  C_*
  =
  \|\phi\|_{L^\infty}
  +
  \frac{K_2}{2a_*b_\lambda(1-\theta)}
  (4c_0)^p
  (c_0m_\phi)^{1-p}
  \|\phi'\|_{L^p}^p .
\]
Assume that
\[
  K_1\ge 2C_*.
\]
Then there exists \(\delta_3>0\), depending only on \(c\) and \(\phi\), such that if
\(0<\delta\le\delta_3\), then
\[
  \sup_{0\le t\le T}\|u(t)\|_{L^\infty}
  \le
  \frac{K_1}{2}\delta .
\]
\end{proposition}

\begin{proof}
By Proposition \ref{prop:riccati_P}, we have \eqref{bt-cone} and \eqref{P-gamma-int} on \([0,T]\).
Let \((t,x)\in[0,T]\times\mathbb R\) be fixed, and let \(x_+(s)\) be the
plus characteristic curve satisfying
\[
  x_+(t)=x.
\]
By the definition of \(R\), we have
\[
  \frac{d}{ds}u(s,x_+(s))=R(s,x_+(s)).
\]
Hence
\[
  u(t,x)
  =
  u_0(x_+(0))
  +
  \int_0^t R(s,x_+(s))\,ds.
\]
By Lemma \ref{zz1}, we have \(R(t,x)\le0\). Thus, putting \(r=-R\), we obtain
\[
  |u(t,x)|
  \le
  \|u_0\|_{L^\infty}
  +
  \int_0^t A(s)\,ds.
\]
By Proposition \ref{prop:A_improve},
\[
  A(s)
  \le
  \frac{K_2}{2}E_0P(s)^\theta.
\]
Therefore
\[
  |u(t,x)|
  \le
  \|u_0\|_{L^\infty}
  +
  \frac{K_2}{2}E_0
  \int_0^t P(s)^\theta\,ds.
\]
Using \eqref{P-gamma-int} with \(\gamma=\theta\), we have
\[
  \int_0^t P(s)^\theta\,ds
  \le
  \frac{1}{a_*b_\lambda(1-\theta)}
  P(0)^{\theta-1}.
\]
Hence
\[
  |u(t,x)|
  \le
  \|u_0\|_{L^\infty}
  +
  \frac{K_2}{2a_*b_\lambda(1-\theta)}
  E_0P(0)^{\theta-1}.
\]

For the initial data in Theorem \ref{mainthm},
\[
  \|u_0\|_{L^\infty}
  \le
  \delta\|\phi\|_{L^\infty}.
\]
Choose \(\delta_3>0\), depending only on \(c\) and \(\phi\), so small that
\[
  \frac{c_0}{2}
  \le
  c(\delta\phi(x))
  \le
  2c_0
\]
for all \(x\in\mathbb R\) and \(0<\delta\le\delta_3\). Then
\[
  P(0)
  =
  \sup_{x\in\mathbb R}
  \left\{
    -2c(\delta\phi(x/\varepsilon))
    \frac{\delta}{\varepsilon}
    \phi'(x/\varepsilon)
  \right\}
  \ge
  c_0m_\phi\frac{\delta}{\varepsilon}.
\]
Moreover,
\[
\begin{aligned}
  E_0
  &=
  \int_{\mathbb R}
  \left|
    -2c(\delta\phi(x/\varepsilon))
    \frac{\delta}{\varepsilon}
    \phi'(x/\varepsilon)
  \right|^p\,dx  \\
  &\le
  (4c_0)^p
  \left(\frac{\delta}{\varepsilon}\right)^p
  \int_{\mathbb R}
  |\phi'(x/\varepsilon)|^p\,dx  \\
  &=
  (4c_0)^p
  \|\phi'\|_{L^p}^p
  \left(\frac{\delta}{\varepsilon}\right)^p
  \varepsilon .
\end{aligned}
\]
Since
\[
  \theta-1=1-p,
\]
we obtain
\[
\begin{aligned}
  E_0P(0)^{\theta-1}
  &\le
  (4c_0)^p
  \|\phi'\|_{L^p}^p
  \left(\frac{\delta}{\varepsilon}\right)^p
  \varepsilon
  \left(
    c_0m_\phi
    \frac{\delta}{\varepsilon}
  \right)^{1-p}  \\
  &=
  (4c_0)^p
  (c_0m_\phi)^{1-p}
  \|\phi'\|_{L^p}^p
  \delta .
\end{aligned}
\]
Therefore
\[
  |u(t,x)|
  \le
  C_*\delta.
\]
Taking the supremum with respect to \((t,x)\in[0,T]\times\mathbb R\), we get
\[
  \sup_{0\le t\le T}\|u(t)\|_{L^\infty}
  \le
  C_*\delta.
\]
By the choice
\[
  K_1\ge 2C_*,
\]
we conclude that
\[
  \sup_{0\le t\le T}\|u(t)\|_{L^\infty}
  \le
  \frac{K_1}{2}\delta.
\]
This completes the proof.
\end{proof}

\subsection{Completion of the proof}

We now complete the proof of Theorem \ref{mainthm}. By Propositions
\ref{prop:E_improve}, \ref{prop:char_improve}, \ref{prop:A_improve}, and
\ref{prop:u_improve}, the bootstrap estimates are improved on \([0,T]\),
provided that \(\delta>0\) is sufficiently small. Since the constants in the
improved estimates are independent of \(T<T^*\), the standard continuity
argument implies that these estimates hold on every compact subinterval of
\([0,T^*)\).

In particular, by Proposition \ref{prop:riccati_P}, we have
\[
  D_-P(t)\ge a_*b_\lambda P(t)^2
\]
for \(0\le t<T^*\). Since \(P\) is locally Lipschitz on \([0,T^*)\), it is
absolutely continuous on every compact subinterval of \([0,T^*)\). Hence
\[
  P'(t)\ge a_*b_\lambda P(t)^2
\]
for almost every \(t\in[0,T^*)\).

Since \(P(0)>0\), it follows that \(P(t)>0\) for \(0\le t<T^*\). Therefore,
for almost every \(t\in[0,T^*)\),
\[
  \frac{d}{dt}\left(\frac1{P(t)}\right)
  =
  -\frac{P'(t)}{P(t)^2}
  \le
  -a_*b_\lambda .
\]
Integrating this inequality over \([0,t]\), we obtain
\[
  \frac1{P(t)}
  \le
  \frac1{P(0)}-a_*b_\lambda t .
\]
Thus
\[
  P(t)
  \ge
  \frac{P(0)}{1-a_*b_\lambda P(0)t}
\]
as long as the right-hand side is finite.

For the initial data in Theorem \ref{mainthm}, taking \(\delta>0\)
sufficiently small, we have
\[
  P(0)\ge c_0m_\phi\frac{\delta}{\varepsilon},
  \qquad
  m_\phi=\sup_{x\in\mathbb R}(-\phi'(x))>0.
\]
Hence the above lower bound implies that \(P(t)\) cannot remain finite beyond
the time
\[
  \frac{1}{a_*b_\lambda P(0)}
  \le
  \frac{1}{a_*b_\lambda c_0m_\phi}
  \frac{\varepsilon}{\delta}.
\]
Consequently,
\[
  T^*
  \le
  C\frac{\varepsilon}{\delta}.
\]
Moreover, Proposition \ref{prop:u_improve} gives
\[
  \sup_{[0,T^*)\times\mathbb R}|u(t,x)|\le C\delta .
\]
Hence, by \eqref{bt-c},
\[
  \frac{c_0}{2}\le c(u(t,x))\le 2c_0
\]
for \((t,x)\in [0,T^*)\times\mathbb R\). Also, since
\[
  u_t=\frac{R+S}{2},
  \qquad
  u_x=\frac{R-S}{2c(u)},
\]
and \(E(t)\le 2E_0\), we obtain
\[
  \int_{\mathbb R}
  \left(|u_t(t,x)|^{2/\lambda}
  +|u_x(t,x)|^{2/\lambda}\right)\,dx
  \le C E_0
  \le C\delta^{2/\lambda}\varepsilon^{1-2/\lambda}.
\]

It remains to identify the type of breakdown. Since \(u\) remains bounded
and \(c(u)\) is bounded away from zero, the continuation criterion for
classical solutions implies that
\[
\limsup_{t\to T^*}
\left(\|u_t(t)\|_{L^\infty}+\|u_x(t)\|_{L^\infty}\right)=\infty .
\]
This completes the proof of Theorem \ref{mainthm}.

\section{Proof of Proposition \ref{prop:holder}}

We first prove the estimate in the time direction. Proceeding as in the derivation of \eqref{lp-identity}, subtracting the identity for
\(|R|^p\) from that for \(|S|^p\), we obtain
\begin{equation}\label{lp-time-identity}
\begin{aligned}
  &\frac1p
  \left\{
    \partial_t(|S|^p-|R|^p)
    +
    \partial_x\left(c(u)(|R|^p+|S|^p)\right)
  \right\} \\
  &\qquad =
  \left(
    \frac12-\frac{\lambda}{4}
  \right)
  \frac{c'(u)}{c(u)}
  (S-R)
  \left\{
    S|R|^{p-2}R
    +
    R|S|^{p-2}S
  \right\}.
\end{aligned}
\end{equation}
As in the proof of Proposition \ref{prop:E_improve}, we have
\[
\begin{aligned}
  &\left|
    (S-R)
    \left\{
      S|R|^{p-2}R
      +
      R|S|^{p-2}S
    \right\}
  \right|  \\
  &\qquad\le
  4
  \left\{
    A(t)r^p
    +
    A(t)^{p-1}P(t)^{2-p}|S|^p
  \right\}.
\end{aligned}
\]
Let \(0\le t_1<t_2<T^*\) and \(x\in\mathbb R\). Integrating
\eqref{lp-time-identity} over \([t_1,t_2]\times(-\infty,x]\), we get
\[
\begin{aligned}
  &\int_{t_1}^{t_2}
  c(u(s,x))
  \left(
    |R(s,x)|^p+|S(s,x)|^p
  \right)\,ds \\
  &\quad\le
  C\{E(t_1)+E(t_2)\} \\
  &\qquad
  +
  C\int_{t_1}^{t_2}
  \left\{
    A(s)+A(s)^{p-1}P(s)^{2-p}
  \right\}E(s)\,ds .
\end{aligned}
\]
By Proposition \ref{prop:E_improve},
\[
  E(s)\le 2E_0.
\]
By Proposition \ref{prop:A_improve},
\[
  A(s)\le C E_0P(s)^\theta.
\]
Thus
\[
  A(s)^{p-1}P(s)^{2-p}
  \le
  C E_0^{p-1}P(s)^{p\theta}.
\]
Using \eqref{P-gamma-int} with \(\gamma=\theta\) and \(\gamma=p\theta\), we obtain
\[
  \int_{t_1}^{t_2}
  \left\{
    |R(s,x)|^p+|S(s,x)|^p
  \right\}\,ds
  \le
  C E_0 .
\]
Since
\[
  u_t=\frac{R+S}{2},
\]
we have
\[
  \int_{t_1}^{t_2}|u_t(s,x)|^p\,ds
  \le
  C E_0.
\]
Therefore, by H\"older's inequality,
\[
\begin{aligned}
  |u(t,x)-u(s,x)|
  &\le
\left|  \int_s^t |u_t(\tau,x)|\,d\tau \right| \\
  &\le
  |t-s|^{1-1/p}
  \left(
    \int_s^t |u_t(\tau,x)|^p\,d\tau
  \right)^{1/p}  \\
  &\le
  C E_0^{1/p}|t-s|^{1-1/p}.
\end{aligned}
\]

Next we prove the estimate in the space direction. Since
\[
  u_x=\frac{R-S}{2c(u)},
\]
and \(c(u)\) is bounded away from zero on \([0,T^*)\), Proposition
\ref{prop:E_improve} gives
\[
  \|u_x(t)\|_{L^p}
  \le
  C
  \left(
    \|R(t)\|_{L^p}
    +
    \|S(t)\|_{L^p}
  \right)
  \le
  C E_0^{1/p}.
\]
Hence
\[
\begin{aligned}
  |u(t,x)-u(t,y)|
  &\le
\left|  \int_y^x |u_x(t,z)|\,dz \right|  \\
  &\le
  |x-y|^{1-1/p}\|u_x(t)\|_{L^p}  \\
  &\le
  C E_0^{1/p}|x-y|^{1-1/p}.
\end{aligned}
\]
Combining the time and space estimates, we obtain
\[
  |u(t,x)-u(s,y)|
  \le
  C E_0^{1/p}
  \left(
    |t-s|^{1-1/p}
    +
    |x-y|^{1-1/p}
  \right).
\]

For the initial data in Theorem \ref{mainthm},
\[
  E_0\le
  C
  \left(\frac{\delta}{\varepsilon}\right)^p
  \varepsilon.
\]
Thus
\[
  E_0^{1/p}
  \le
  C\delta\varepsilon^{1/p-1}
  =
  C\delta\varepsilon^{\lambda/2-1}.
\]
Since
\[
  1-\frac1p=1-\frac{\lambda}{2},
\]
we get
\[
  |u(t,x)-u(s,y)|
  \le
  C\delta\varepsilon^{\lambda/2-1}
  \left(
    |t-s|^{1-\lambda/2}
    +
    |x-y|^{1-\lambda/2}
  \right).
\]

Finally, for each fixed \(x\in\mathbb R\), the time estimate shows that
\(\{u(t,x)\}_{0\le t<T^*}\) is a Cauchy family as \(t\to T^*\). Hence the
limit
\[
  u(T^*,x)=\lim_{t\to T^*}u(t,x)
\]
exists. Letting \(t\to T^*\) in the space estimate, we obtain
\[
  |u(T^*,x)-u(T^*,y)|
  \le
  C\delta\varepsilon^{\lambda/2-1}
  |x-y|^{1-\lambda/2}.
\]
This completes the proof.

\bibliographystyle{amsplain}
\bibliography{bibliography}

\end{document}